\documentclass[3p]{my_elsarticle}

\usepackage{amsmath,amssymb,amsthm}

\newtheorem{theorem}{Theorem}
\newtheorem{lemma}[theorem]{Lemma}
\newtheorem{definition}[theorem]{Definition}
\newtheorem{remark}{Remark}

\AtBeginDocument{{\noindent\small
This is a preprint of a paper whose final and definite form 
is in 'Computers and Mathematics with Applications', ISSN: 0898-1221.
Submitted 14-Dec-2015; Revised 16-May-2016; Accepted 25-May-2016.}
\vspace{9mm}}

\begin{document}

\begin{frontmatter}

\title{Galerkin Spectral Method for the Fractional Nonlocal Thermistor Problem}

\author[mrsa]{Moulay Rchid Sidi Ammi}
\ead{sidiammi@ua.pt}

\author[dfmt]{Delfim F. M. Torres\corref{cor}}
\ead{delfim@ua.pt}

\cortext[cor]{Corresponding author.}

\address[mrsa]{Department of Mathematics, AMNEA Group,
Faculty of Sciences and Techniques,\\
Moulay Ismail University, B.P. 509, Errachidia, Morocco}

\address[dfmt]{Center for Research and Development
in Mathematics and Applications (CIDMA),\\
Department of Mathematics, University of Aveiro, 3810-193 Aveiro,
Portugal}


\begin{abstract}
We develop and analyse a numerical method for the time-fractional
nonlocal thermistor problem. By rigorous proofs, some error
estimates in different contexts are derived, showing that the
combination of the backward differentiation in time and the
Galerkin spectral method in space leads, for an enough smooth
solution, to an approximation of exponential convergence in space.
\end{abstract}

\end{frontmatter}


\section{Introduction}

Fractional derivatives express properties of memory and
heredity of materials, which is their main benefit when compared
with integer-order derivatives. Practical problems require
definitions of fractional derivatives that allow the use of
physically interpretable initial conditions. Fractional time
derivatives are linked with irregular sub-diffusion, where a
darken of particles spread slower than in classical diffusion. 
The fractional space derivatives are used to model irregular 
diffusion or dispersion, where a particle spreads at a rate 
that does not agree with the classical Brownian motion, and 
the follow can be asymmetric \cite{MR3071220}.

Fractional differential and integro-differential equations occur
in different real processes and physical phenomena, such as in
signal processing and image processing, optics, engineering,
control theory, computer science (such as real neural networks,
complex neural networks and information technology), statistics 
and probability, astronomy, geophysics, hydrology, chemical
technology, materials, robots, earthquake analysis, electric
fractal network, statistical mechanics, biotechnology, medicine,
and economics \cite{kilbas,pod,samko,MR3287248}.

In this paper, we consider the problem of the nonlocal
time-fractional thermistor problem. This fractional model is
obtained from the integer order one
\begin{equation}
\label{eq1} 
\frac{\partial u(x, t)}{\partial t}- \triangle u 
= \frac{\lambda f(u)}{\left( \int_{\Omega} f(u)\, dx\right)^{2}}\, ,
\mbox{ in } Q_{T} = \Omega \times (0, T),
\end{equation}
by replacing the derivative term by a fractional derivative of
order $\alpha > 0$:
\begin{equation}
\label{eq3}
\begin{gathered}
\frac{\partial^{\alpha} u}{\partial t^{\alpha}}- \triangle u
=\frac{\lambda f(u)}{\left( \int_{\Omega} f(u)\, dx\right)^{2}},
\quad \text{ in } Q_{T}= \Omega \times (0, T), \\
\frac{\partial u}{\partial n} = 0,
\quad \mbox{ on } S_{T} = \partial \Omega \times (0, T),\\
u(0)= u_{0}, \quad \text{ in } \Omega,
\end{gathered}
\end{equation}
where $\frac{\partial^{\alpha} u(x, t)}{\partial t^{\alpha}}$
denotes the Caputo fractional derivative of order $\alpha$, 
$0<\alpha < 1$, as defined in \cite{gorenflo} and given by
\[
\frac{\partial^{\alpha} u(x, t)}{\partial t^{\alpha}} =
\frac{1}{\Gamma(1-\alpha)} \int_{0}^{t} \frac{\partial u(x,
s)}{\partial s} \frac{ds}{(t-s)^{\alpha}}, \quad 0 < \alpha < 1,
\]
with $\triangle$ the Laplacian with respect to the spacial
variables and where $f$ is assumed to be a smooth function, 
as prescribed below, and $T$ is a fixed positive real. Here $n$
denotes the outward unit normal and $\frac{\partial}{\partial n}
= n \cdot \nabla$ is the normal derivative on $\partial \Omega$.
Such problems arise in many applications, for instance, in
studying the heat transfer in a resistor device whose electrical
conductivity $f$ is strongly dependent on the temperature $u$.
Constant $\lambda$ is a dimensionless parameter, which can be
identified with the square of the applied potential difference at
the ends of the conductor. Function $u$ represents the temperature
generated by the electric current flowing through a conductor.

A fractional order model instead of its classical integer order
counterpart has been considered here because fractional order
differential equations are generalizations of integer order
differential equations and fractional order models possess memory.
Moreover, the fact that resistors are influenced by memory makes
fractional modelling appropriate for this kind of dynamical
problems. We use Caputo's definition. The main advantage is that
the initial conditions for fractional differential equations with
Caputo derivatives take the same form as for integer-order
differential equations. Note that \eqref{eq3} covers \eqref{eq1}
and extends it to more general cases. The classical nonlocal
thermistor problem \eqref{eq1} with the time derivative of integer
order can be obtained by taking the limit $\alpha \rightarrow 1$
in \eqref{eq3} (see \cite{yumin}), while the case $\alpha =0$ 
corresponds to the steady state thermistor problem. 
In the case $0 < \alpha < 1$, the Caputo fractional 
derivative depends on and uses the information
of the solutions at all previous time levels (non-Markovian
process). In this case the physical interpretation of fractional
derivative is that it represents a degree of memory in the
diffusing material. Such kind of models have been analytically
investigated by a number of authors, using Green functions, the
Laplace and Fourier--Laplace transform methods, in order to
construct analytical solutions. However, papers in the literature
on the numerical solutions of time fractional differential
equations are still under development. In \cite{sidiammi1},
existence and uniqueness of a positive solution to a generalized
spatial fractional-order nonlocal thermistor problem is proved.
Stability and error analysis of the semi-discretized fractional
nonlocal thermistor problem is investigated in
\cite{sidiammi2,sidiammi3}. More precisely, in
\cite{sidiammi2,sidiammi3} a finite difference method is proposed,
respectively for solving the semidiscretized fractional nonlocal
thermistor problem and the time fractional thermistor problem,
which is a system of elliptic-parabolic PDEs and where some
stability as well as error analysis for this scheme are derived
for both problems. Herein, an approach based on finite differences
combined with the Galerkin spectral method is used to solve the
nonlocal time fractional thermistor problem. By definition of
fractional derivative, to compute the solution at the current time
level one needs to save all the previous solutions, which makes
the storage expensive if low-order methods are employed for
spatial discretization. One of the main advantage of the spectral
method is the fact that it can relax this storage limit since it
needs fewer grid points to produce a highly accurate solution
\cite{MR3367137,MR3359783}.

The text is organized as follows. In Section~\ref{sec2} a finite
difference scheme for the temporal discretization of problem
\eqref{eq3} is introduced. Then, in Section~\ref{sec3}, we provide
a finite difference-Galerkin spectral method to obtain error
estimates of $(2-\alpha)$-order convergence in time and
exponential convergence in space, for smooth enough solutions. 
The proof of our main result (Theorem~\ref{thm:ea}) is given in
Section~\ref{sec4}. Finally, in Section~\ref{sec5} we carry out 
an error analysis between the solution $u_{N}^{k}$ of the full
discretized problem and the exact solution $u$.
We end with Section~\ref{sec:conc} of conclusions
and future work.


\section{Time discretization: A finite difference scheme}
\label{sec2}

Several theoretical analysis, on various aspects 
of both steady-state and time-dependent thermistor equations, 
with different aspects and types of boundary and initial conditions, 
have been carried out in the literature. For existence of weak solutions, 
uniqueness and related regularity and smoothness results, 
in several settings and under different assumptions on the coefficients, 
we refer the reader to \cite{antontsev}. For our purposes,
the $L^{\infty}$-energy method is a suitable and powerful 
tool to prove existence, regularity, and uniqueness 
of solutions to \eqref{eq3}. From the results of \cite{sidiammi4}, 
it follows by the $L^{\infty}$-energy method that 
problem \eqref{eq3} has a unique and sufficiently smooth 
solution under the following assumptions:
\begin{itemize}
\item[(H1)] $f: \mathbb{R} \rightarrow \mathbb{R}$ is a positive
Lipshitz and $\mathbb{C}^{1}$ continuous function;

\item[(H2)] there exist positive constants $c$ and $\beta$ such
that for all $\xi \in \mathbb{R}$ we have $c \leq f(\xi) \leq c
|\xi|^{\beta +1} +c$;

\item[(H3)] $u_{0} \in W^{1,\infty}(\Omega)$.
\end{itemize}
Let $\|\cdot\|_{0}$ be the $L^{2}$ norm. It can be shown 
(see, e.g., \cite{zidler}) that the quantity
\begin{equation}
\label{equiv} 
\|v\|_{1}= \left(\|v\|_{0}^{2}+\alpha_{0}\|
\frac{du}{dx}\|_{0}^{2}\right)^{\frac{1}{2}},
\end{equation}
where $\alpha_{0}$ is given below, defines a norm on
$H^{1}(\Omega)$, which is equivalent to the
$\|\cdot\|_{H^{1}(\Omega)}$ norm. Note that 
$\|\cdot\|_{m}$, $m>1$, is the $H^{m}$ norm.

We introduce a finite difference approximation to discretize the
time-fractional derivative. Let $\delta = \frac{T}{N}$ be the
length of each time step, for some large $N$, and $t_{k}= k
\delta$, $k=0, 1, \ldots , K$. We use the following formulation:
for all $0 \leq k \leq K-1$,
\begin{equation}
\label{eq4}
\begin{split}
\frac{\partial^{\alpha} u(x, t)}{\partial t^{\alpha}} &=
\frac{1}{\Gamma(1-\alpha)} \sum_{j=0}^{k}\int_{t_{j}}^{t_{j+1}}
\frac{\partial u(x, s)}{\partial s} \frac{ds}{(t_{k+1}-s)^{\alpha}}\\
&= \frac{1}{\Gamma(1-\alpha)} \sum_{j=0}^{k} \frac{u(x,t_{j+1})
-u(x, t_{j})}{\delta}\int_{t_{j}}^{t_{j+1}}
\frac{ds}{(t_{k+1}-s)^{\alpha}}+ r_{\delta}^{k+1},
\end{split}
\end{equation}
where $r_{\delta}^{k+1}$ is the truncation error. It can be seen
from \cite{yumin} that the truncation error verifies
\begin{equation}
\label{eq5} 
r_{\delta}^{k+1} \lesssim c_{u} \delta^{2-\alpha},
\end{equation}
where $c_{u}$ is a constant depending only on $u$. 
On the other hand, by change of variables, we have
\begin{equation*}
\begin{split}
\frac{1}{\Gamma(1-\alpha)} \sum_{j=0}^{k} &\frac{u(x,
t_{j+1})-u(x, t_{j})}{\delta}
\int_{t_{j}}^{t_{j+1}}\frac{ds}{(t_{k+1}-s)^{\alpha}}\\
&= \frac{1}{\Gamma(1-\alpha)} \sum_{j=0}^{k} \frac{u(x,
t_{j+1})-u(x, t_{j})}{\delta}
\int_{t_{k-j}}^{t_{k+1-j}} \frac{dt}{t^{\alpha}}\\
&= \frac{1}{\Gamma(1-\alpha)} \sum_{j=0}^{k} \frac{u(x,
t_{k+1-j})-u(x, t_{k-j})}{\delta}\int_{t_{j}}^{t_{j+1}}
\frac{dt}{t^{\alpha}}\\
&= \frac{1}{\Gamma(2-\alpha)} \sum_{j=0}^{k} \frac{u(x,t_{k+1-j})
-u(x, t_{k-j})}{\delta^{\alpha}}\left\{
(j+1)^{1-\alpha}-(j)^{1-\alpha}\right\}.
\end{split}
\end{equation*}
Let us denote $b_{j}:=(j+1)^{1-\alpha}- j^{1-\alpha}$,
$j=0, 1, \ldots k$. Note that
\begin{equation}
\label{eq12}
\begin{gathered}
b_{j} >0, \quad j=0, 1, \ldots k,\\
1=b_{0} > b_{1} > \ldots > b_{k}, \quad b_{k} \rightarrow 0 
\text{ as } k \rightarrow \infty,\\
\sum_{j=0}^{k}(b_{j}-b_{j+1}) + b_{k+1} = (1-b_{1})
+ \sum_{j=1}^{k-1}(b_{j}-b_{j+1}) + b_{k}=1.
\end{gathered}
\end{equation}
Define the discrete fractional differential 
operator $L^{\alpha}_{t}$ by
$$
L^{\alpha}_{t}u(x, t_{k+1}) = \frac{1}{\Gamma(2-\alpha)}
\sum_{j=0}^{k} b_{j} \frac{u(x, t_{k+1-j})-u(x,
t_{k-j})}{\delta^{\alpha}}.
$$
Then \eqref{eq4} becomes
\begin{equation*}
\frac{\partial^{\alpha} u(x,t_{k+1})}{\partial t^{\alpha}}
= L^{\alpha}_{t}u(x, t_{k+1}) + r_{\delta}^{k+1}.
\end{equation*}
Using this approximation, we arrive to the following 
finite difference scheme to \eqref{eq3}:
\begin{equation}
\label{eq8}
\begin{gathered}
L^{\alpha}_{t}u^{k+1}(x)-\triangle u^{k+1} 
= \frac{\lambda f(u^{k+1})}{( \int_{\Omega} f(u^{k+1})\, dx)^{2}}
\quad \text{ in }  \Omega,
\end{gathered}
\end{equation}
$k=1, \ldots, K-1$, where $u^{k+1}(x)$ are approximations 
to $u(x,t_{k+1})$. Scheme \eqref{eq8} can be reformulated as
\begin{equation}
\label{eq9}
\begin{split}
b_{0}u^{k+1}-\Gamma(2-\alpha)\delta^{\alpha}\triangle u^{k+1}
&=b_{0}u^{k}- \sum_{j=1}^{k}b_{j}\{ u^{k+1-j}-u^{k-j} \}
+\Gamma(2-\alpha)\delta^{\alpha} \frac{\lambda f(u^{k+1})}{\left(
\int_{\Omega} f(u^{k+1})\, dx\right)^{2}}\\
&=b_{0}u^{k} - \sum_{j=0}^{k-1}b_{j+1} u^{k-j}
+\sum_{j=1}^{k}b_{j} u^{k-j}+ \Gamma(2-\alpha)\delta^{\alpha}
\frac{\lambda f(u^{k+1})}{( \int_{\Omega} f(u^{k+1})\, dx)^{2}}\\
&=b_{0}u^{k} + \sum_{j=0}^{k-1}(b_{j}-b_{j+1}) u^{k-j}
+\Gamma(2-\alpha)\delta^{\alpha}\frac{\lambda f(u^{k+1})}{\left(
\int_{\Omega} f(u^{k+1}\right)\, dx)^{2}}.
\end{split}
\end{equation}
To complete the semi-discrete problem, 
we consider the boundary conditions
\begin{equation}
\label{equa11}
\begin{gathered}
\frac{\partial u^{k+1}}{\partial n}=0
\end{gathered}
\end{equation}
and the initial condition $u^{0}=u_{0}$. If we set $\alpha_{0}
:= \Gamma(2-\alpha) \delta^{\alpha}$, then \eqref{eq9} can be
rewritten in the form
\begin{equation}
\label{eq11} 
u^{k+1}-\alpha_{0} \triangle u^{k+1}= (1-b_{1})u^{k}
+\sum_{j=1}^{k-1}(b_{j}-b_{j+1})u^{k-j}+b_{k}u^{0}+ \alpha_{0}
\frac{\lambda f(u^{k+1})}{( \int_{\Omega} f(u^{k+1})\, dx)^{2}}
\end{equation}
for all $k \geq 1$. When $k=0$, scheme \eqref{eq11} reads
\begin{equation*}
u^{1}-\alpha_{0} \triangle u^{1}= u^{0}+ \alpha_{0}\frac{\lambda
f(u^{1})}{\left(\int_{\Omega} f(u^{1})\, dx\right)^{2}};
\end{equation*}
when $k=1$, scheme \eqref{eq11} becomes
\begin{equation*}
u^{2}-\alpha_{0} \triangle u^{2}= (1-b_{1})u^{1}+b_{1}u^{0}
+\alpha_{0}\frac{\lambda f(u^{2})}{\left(\int_{\Omega}
f(u^{2})\,dx\right)^{2}}.
\end{equation*}
Define the error term $r^{k+1}$ by
\begin{equation*}
r^{k+1} := \alpha_{0} \left\{\frac{\partial^{\alpha}
u(x,t_{k+1})}{
\partial t^{\alpha}} - L_{t}^{\alpha}u(x, t_{k+1})\right\}.
\end{equation*}
Then we get from \eqref{eq5} that
\begin{equation}
\label{eq12bis} 
|r^{k+1}| = \Gamma(2-\alpha) \delta^{\alpha}
|r_{\delta}^{k+1}| \leq c_{u} \delta^{2}.
\end{equation}
Our aim is now to define the weak formulation of \eqref{eq8}.

\begin{definition}
We say that $u^{k+1}$ is a weak solution of \eqref{eq8} if
\begin{equation}
\label{eq8bis}
\left(u^{k+1},v\right)+\alpha_{0}\int_{\Omega}\nabla u^{k+1}\nabla
v \,dx =\left(f^{k}, v\right)+ \alpha_{0}\frac{\lambda
f(u^{k+1})}{\left( \int_{\Omega} f\left(u^{k+1}\right)\,
dx\right)^{2}},
\end{equation}
where $f^{k}= (1-b_{1})u^{k}
+\sum_{j=1}^{k-1}(b_{j}-b_{j+1})u^{k-j}+b_{k}u^{0}$.
\end{definition}


\section{A Galerkin spectral method in space}
\label{sec3}

Let $\Omega=(-1, 1)$. We define $\mathbb{P}_{N}(\Omega)$ to be the
space of all polynomials of degree $\leq N$ with respect to space
$x$. Then, denote $\mathbb{P}_{N}^{0}(\Omega):=H_{0}^{1}(\Omega)
\bigcap \mathbb{P}_{N}(\Omega)$. The Galerkin method is of
interest in its own right. It offers some advantages in numerical
analysis, and could be implemented once a suitable basis for the
space $\mathbb{P}_{N}^{0}$ is chosen. It consists in approximating
the solution by polynomials of high degree. Let the spectral 
discretization of problem \eqref{eq8bis} be defined as follows: 
find $u_{N}^{k+1} \in \mathbb{P}_{N}^{0}(\Omega)$ such that 
for all $v_{N} \in \mathbb{P}_{N}^{0}(\Omega)$
\begin{equation}
\label{eq9bis} 
\left(u^{k+1}_{N},v_{N}\right)+\alpha_{0}
\int_{\Omega}\nabla u^{k+1}_{N}\nabla v_{N} \,dx =\left(f_{N}^{k},
v_{N}\right)+ \frac{\lambda f(u^{k+1}_{N})}{\left( \int_{\Omega}
f(u^{k+1}_{N})\, dx\right)^{2}},
\end{equation}
where
$$
f_{N}^{k}= (1-b_{1})u^{k}_{N}+\sum_{j=1}^{k-1}
\left(b_{j}-b_{j+1}\right)u^{k-j}_{N}+b_{k}u^{0}_{N}.
$$
Thanks to the classical theory of elliptic problems, the
well-posedeness of problem \eqref{eq9bis} is immediate for given
$\{u_{N}^{j}\}_{j=0}^{k}$. Now our main goal is to derive an error
estimate for the full-discrete solution $\{u_{N}^{k}\}_{k=0}^{K}$.
Let $\pi_{N}^{1}$ be the $H^{1}$-orthogonal projection operator
from $H_{0}^{1}(\Omega)$ into $\mathbb{P}_{N}^{0}(\Omega)$ defined
as follows: for all $\psi \in H^{1}_{0}(\Omega)$, $\pi_{N}^{1}\psi
\in \mathbb{P}_{N}^{0}(\Omega)$, such that
\begin{equation}
\label{eq14bis} 
\left(\pi_{N}^{1}\psi ,v_{N}\right)
+\alpha_{0}\int_{\Omega}\nabla \pi_{N}^{1}\psi \nabla v_{N} \,dx
=\left(\psi ,v_{N}\right)+\alpha_{0}\int_{\Omega}\nabla \psi\nabla
v_{N} \,dx
\end{equation}
for all $v_{N} \in \mathbb{P}_{N}^{0}(\Omega)$. 
We recall the following projection estimate.

\begin{lemma}[See \cite{bernardi}]
\label{lm2} 
If $\psi \in H^{m}(\Omega)\bigcap H_{0}^{1}(\Omega)$,
$m\geq 1$, then
$$
\left\|\psi -\pi_{N}^{1}\psi \right\|_{1} \leq cN^{1-m}
\|\psi\|_{m}.
$$
\end{lemma}

We carry out an error analysis between the solution $u_{N}^{k}$ of
the full discretized problem and the solution $u^{k}$ of the
semi-discretized problem.

\begin{theorem}
\label{thm:ea} 
Let $\{u_{N}^{k}\}_{k=0}^{K}$ be the solution of
problem \eqref{eq9bis} with $u_{N}^{0}= \pi_{N}^{1}u^{0}$ the
initial condition. Further, suppose that $u^{k} \in
H^{m}(\Omega)\bigcap H_{0}^{1}(\Omega)$, $m > 1$. 
Then the following error estimates hold:
\begin{itemize}
\item[(a)]
$$
\|u^{k}-u^{k}_{N}\|_{1} \leq \frac{c T^{\alpha}}{1-\alpha}
\delta^{-\alpha} N^{1-m} \max_{0 \leq j \leq k}\|u^{j}\|_{m},
\quad k=1, \ldots, K,
$$
where $0\leq \alpha < 1$ and $c$ is a positive constant;

\item[(b)] if $\alpha \rightarrow 1$, then
$$
\|u^{k}-u^{k}_{N}\|_{1}  \leq c \delta^{-1} N^{1-m} \sum_{j=0}^{k}
\delta \|u^{j}\|_{m}, \quad k=1, \ldots, K,
$$
with $c$ a constant depending only on $T$.
\end{itemize}
\end{theorem}


\section{Proof of Theorem~\ref{thm:ea}}
\label{sec4}

By the definition \eqref{eq14bis} of $\pi_{N}^{1}$, we have
\begin{equation}
\label{eq15bis}
\begin{split}
\left(\pi_{N}^{1}u^{k+1} , v_{N}\right) &+
\alpha_{0}\int_{\Omega}\nabla \pi_{N}^{1}u^{k+1} \nabla v_{N} \,dx
= \left(1-b_{1}\right)\left(u^{k}, v_{N}\right)\\
&+ \sum_{j=1}^{k-1}(b_{j}-b_{j+1})(u^{k-j}, v_{N}) 
+b_{k}(u^{0}, v_{N}) +\lambda \alpha_{0}\left(
\frac{f(\pi_{N}^{1}u^{k+1})}{\left( \int_{\Omega}
f(\pi_{N}^{1}u^{k+1})\,dx\right)^{2}}, v_{N}\right)
\end{split}
\end{equation}
for all $v_{N} \in \mathbb{P}_{N}^{0}(\Omega)$. Let
$\widetilde{e}_{N}^{k+1}=\pi_{N}^{1}u^{k+1}-u_{N}^{k+1},
e_{N}^{k+1}= u^{k+1}-u_{N}^{k+1}$ and set $a_{k}=1-b_{1}$,
$a_{k-j}= b_{j}-b_{j+1}$, $j=1\ldots, k-1$, $a_{0}=b_{k}$.
Subtracting \eqref{eq9bis} from \eqref{eq15bis}, we get
\begin{multline}
\label{eq14} 
\left(\widetilde{e}_{N}^{k+1},v_{N}\right)
+\alpha_{0} \left( \nabla \widetilde{e}_{N}^{k+1}, \nabla v_{N}\right)
= a_{k}\left(e_{N}^{k},v_{N}\right)+\sum_{j=1}^{k-1}a_{k-j}(e_{N}^{k-j},
v_{N})+ a_{0}\left(e_{N}^{0}, v_{N}\right)\\
+ \alpha_{0}\left(\frac{\lambda
f\left(\pi_{N}^{1}u^{k+1}\right)}{\left( \int_{\Omega}
f(\pi_{N}^{1}u^{k+1})\,dx\right)^{2}}, v_{N}\right) 
- \alpha_{0}\left(\frac{\lambda f(u^{k+1})}{\left( 
\int_{\Omega} f(u^{k+1})\, dx\right)^{2}}, v_{N}\right).
\end{multline}
Taking $v_{N}= \widetilde{e}_{N}^{k+1}$ in \eqref{eq14}, we obtain
\begin{multline}
\label{eq16bis} 
\|\widetilde{e}_{N}^{k+1}\|_{1}^{2} \leq \left
\{a_{k} \|e_{N}^{k+1}\|_{0}+\sum_{j=1}^{k-1}a_{k-j}
\|e_{N}^{k-j}\|_{0}+a_{0} \|e_{N}^{0}\|_{0} \right \}
\|\widetilde{e}_{N}^{k+1}\|_{1}\\
+ \alpha_{0}\left(\frac{\lambda f(\pi_{N}^{1}u^{k+1})}{\left(
\int_{\Omega} f(\pi_{N}^{1}u^{k+1})\, dx\right)^{2}},
\widetilde{e}_{N}^{k+1}\right) - \alpha_{0}\left(\frac{\lambda
f(u^{k+1})}{\left( \int_{\Omega} f(u^{k+1})\, dx\right)^{2}},
\widetilde{e}_{N}^{k+1}\right).
\end{multline}
To continue the proof, we shall need the following lemma.

\begin{lemma}[See \cite{sidiammi2}]
\label{lm46} 
Let $u_{i}$, $i=1,2$, be two weak solutions of \eqref{eq3}. 
Assume that $(H1)$--$(H3)$ hold. Then,
\begin{equation*}
\left(\frac{\lambda f(u_{1})}{( \int_{\Omega} f(u_{1})\, dx)^{2}},
w\right) - \left( \frac{\lambda f(u_{2})}{( \int_{\Omega}
f(u_{2})\, dx)^{2}}, w\right) \leq c \|w\|_{2}^{2},
\end{equation*}
where $w=u_{1}-u_{2}$ and $c$ is a positive constant.
\end{lemma}

\noindent Using \eqref{eq16bis}, we get
\begin{equation}
\label{eq1bis} 
\|\widetilde{e}_{N}^{k+1}\|_{1}^{2} \leq \left \{
a_{k} \|e_{N}^{k}\|_{0}+\sum_{j=1}^{k-1}a_{k-j}
\|e_{N}^{k-j}\|_{0}+a_{0} \|e_{N}^{0}\|_{0} \right \}
\|\widetilde{e}_{N}^{k+1}\|_{1} +  c
\|\widetilde{e}_{N}^{k+1}\|_{0}^{2}.
\end{equation}
From Young's inequality, we get
$$
\|\widetilde{e}_{N}^{k+1}\|_{1}^{2} \leq  (c+\varepsilon)
\|\widetilde{e}_{N}^{k+1}\|_{1}^{2} +c_{\varepsilon}\left \{ a_{k}
\|e_{N}^{k}\|_{0}+\sum_{j=1}^{k-1}a_{k-j}
\|e_{N}^{k-j}\|_{0}+a_{0} \|e_{N}^{0}\|_{0} \right \}^{2}
$$
for $c$, $c_{\varepsilon}$ and $\varepsilon$ positive constants.
Hence,
$$
\left(1-(c+\varepsilon)\right) \|\widetilde{e}_{N}^{k+1}\|_{1}^{2}
\leq c_{\varepsilon}\left \{ a_{k}
\|e_{N}^{k}\|_{0}+\sum_{j=1}^{k-1}a_{k-j}
\|e_{N}^{k-j}\|_{0}+a_{0} \|e_{N}^{0}\|_{0} \right \}^{2}.
$$
For a suitable choice of $\varepsilon$, we get
$$
\|\widetilde{e}_{N}^{k+1}\|_{1} \leq c \left \{ a_{k}
\|e_{N}^{k}\|_{0}+\sum_{j=1}^{k-1}a_{k-j}
\|e_{N}^{k-j}\|_{0}+a_{0} \|e_{N}^{0}\|_{0} \right \}
$$
with $c$ a positive constant. We also have, 
by the triangular inequality, that
$$
\|{e}_{N}^{k+1}\|_{1} \leq \|\widetilde{e}_{N}^{k+1}\|_{1}
+\|u^{k+1}-\pi_{N}^{1}u^{k+1}\|_{1}.
$$
Then,
\begin{equation}
\label{eq15}
\begin{split}
\|e_{N}^{k+1}\|_{1}  \leq c \left( a_{k}
\|e_{N}^{k}\|_{0}+\sum_{j=1}^{k-1}a_{k-j}
\|e_{N}^{k-j}\|_{0}+a_{0} \|e_{N}^{0}\|_{0}\right)
+\|u^{k+1}-\pi_{N}^{1}u^{k+1}\|.
\end{split}
\end{equation}
We finish the proof of Theorem~\ref{thm:ea}
by distinguishing the two cases of $\alpha$ 
and proving the necessary estimates.

\begin{lemma}
\label{lemma:5} 
$(i)$ If $0\leq \alpha < 1$, then
\begin{equation}
\label{eq16} 
\|e_{N}^{i}\|_{1} \leq c b^{-1}_{i-1}\max_{0\leq j
\leq k} \|u^{j}-\pi_{N}^{1}u^{j}\|_{1}, \quad i=1, 2, \ldots, K.
\end{equation}
$(ii)$ If $\alpha \rightarrow 1$, then
\begin{equation}
\label{eq17} 
\|e_{N}^{i}\|_{1} \leq c
\sum_{j=0}^{k}\|u^{j}-\pi_{N}^{1}u^{j}\|_{1}, 
\quad i=1, 2,
\ldots, K.
\end{equation}
\end{lemma}

\begin{proof}
$(i)$ By \eqref{eq15}, inequality \eqref{eq16} is obvious for $i=1$. 
Suppose now that \eqref{eq16} holds for $i=1,2, \ldots,k$.
We prove that it remains true for $i=k+1$. By \eqref{eq15}, the
induction hypothesis, and the fact that $(b_{j}^{-1})_{j}$ is an
increasing sequence $\left(b_{j}^{-1} \leq b_{j+1}^{-1}\right)$,
we have, because $a_{0}=b_{k}$ and $\sum_{j=0}^{k}a_{j} =1$, that
\begin{equation}
\begin{split}
\|e_{N}^{k+1}\|_{1} & \leq c \left(a_{k}
\|e_{N}^{k}\|_{0}+\sum_{j=1}^{k-1}a_{k-j}
\|e_{N}^{k-j}\|_{0}+a_{0} \|e_{N}^{0}\|_{0}\right)
+\|u^{k+1}-\pi_{N}^{1}u^{k+1}\|_{1}\\
& \leq c \left( a_{k}b^{-1}_{k-1} +\sum_{j=1}^{k-1}a_{k-j}
b^{-1}_{k-1}\right) \max_{0 \leq j \leq
k}\|u^{j}-\pi_{N}^{1}u^{j}\|_{1}+
\|u^{k+1}-\pi_{N}^{1}u^{k+1}\|_{1}\\
& \leq c \left( a_{k} +\sum_{j=1}^{k-1}a_{k-j} +b_{k}\right)
b_{k}^{-1} \max_{0 \leq j \leq k+1}\|u^{j}-\pi_{N}^{1}u^{j}\|_{1},\\
& \leq c \left( a_{k} +\sum_{j=1}^{k-1}a_{k-j} +a_{0}\right)
b_{k}^{-1} \max_{0 \leq j \leq k+1}\|u^{j}-\pi_{N}^{1}u^{j}\|_{1}\\
& \leq c  b_{k}^{-1} \max_{0 \leq j \leq
k+1}\|u^{j}-\pi_{N}^{1}u^{j}\|_{1}.
\end{split}
\end{equation}
The estimate \eqref{eq16} is proved. Then,
\begin{equation}
\begin{split}
\|e_{N}^{k}\|_{1} & \leq c b_{k-1}^{-1} \max_{0 \leq j \leq
k}\|u^{j}-\pi_{N}^{1}u^{j}\|_{1}\\
& \leq c k^{-\alpha} b_{k-1}^{-1} k^{\alpha} \max_{0 \leq j \leq
k}\|u^{j}-\pi_{N}^{1}u^{j}\|_{1}\\
& \leq c k^{-\alpha} b_{k-1}^{-1}
\delta^{-\alpha}(k\delta)^{\alpha}
\max_{0 \leq j \leq k}\|u^{j}-\pi_{N}^{1}u^{j}\|_{1}\\
& \leq  \frac{c T^{\alpha}}{1-\alpha} \delta^{-\alpha}  N^{1-m}
\max_{0 \leq j \leq k}\|u^{j}\|_{m},
\end{split}
\end{equation}
$1 \leq k \leq K$, where we have used in the above inequalities
the definition of $b_{k}$ and the fact that
$$
k\delta \leq T, k^{-\alpha} b^{-1}_{k-1} \leq \frac{1}{1-\alpha},
\quad k=1, 2, \ldots, K,
$$
which can be obtained by direct calculations. 
$(ii)$ Now, we consider the case $\alpha \rightarrow 1$. 
Again, we proceed by mathematical induction. 
The estimate \eqref{eq17} is easier to prove for $i=1$ 
using \eqref{eq15}. Suppose now that \eqref{eq17}
holds for all $i=1,\ldots,k$. We prove it is also true 
for $i=k+1$. By \eqref{eq15}, we have
\begin{equation*}
\begin{split}
\|e_{N}^{k+1}\|_{1}  &\leq c \left(a_{k}
\|e_{N}^{k}\|_{0}+\sum_{j=1}^{k-1}a_{k-j}
\|e_{N}^{k-j}\|_{0}+a_{0} \|e_{N}^{0}\|_{0} \right)
+\|u^{k+1}-\pi_{N}^{1}u^{k+1}\|_{1}\\
& \leq c \left( a_{k} +\sum_{j=1}^{k-1}a_{k-j} +a_{0} \right)
\sum_{j=0}^{k}\|u^{j}-\pi_{N}^{1}u^{j}\|_{1}+
\|u^{k+1}-\pi_{N}^{1}u^{k+1}\|_{1}\\
& \leq c \sum_{j=0}^{k+1}\|u^{j}-\pi_{N}^{1}u^{j}\|_{1}.
\end{split}
\end{equation*}
Inequality \eqref{eq17} is now derived. 
Therefore, by Lemma~\ref{lm2}, we have
\begin{equation*}
\|e_{N}^{k}\|_{1} \leq
\sum_{j=0}^{k}\|u^{j}-\pi_{N}^{1}u^{j}\|_{1} \leq c
\delta^{-1}N^{1-m} \sum_{j=0}^{k} \delta \|u^{j}\|_{m}.
\end{equation*}
This ends the proof of Lemma~\ref{lemma:5}
and the proof of Theorem~\ref{thm:ea}.
\end{proof}

\begin{remark}
The sum $\sum_{j=0}^{k} \delta \|u^{j}\|_{m}$ 
is the analogous discrete form 
of $\int_{0}^{t_{k}}  \|u(t)\|_{m} dt$.
\end{remark}


\section{Error estimate between the solution of the full
discretized problem and the exact one}
\label{sec5}

Our aim now is to derive an estimate for
$\|u(t_{k})-u^{k}_{N}\|_{1}$.

\begin{theorem}
\label{thm6}
Let $u$ be the exact solution of \eqref{eq3},
and $(u^{k}_{N})_{k=0}^{K}$ be the solution of \eqref{eq9bis} with
the initial condition $u_{N}^{0}=\pi_{N}^{1}u^{0}$. Suppose that
$u$ is regular enough such that $u \in H^{1}([0, T]$,
$H^{m}(\Omega)\bigcap H_{0}^{1}(\Omega))$, $m>1$. 
Then the following error estimates hold:
\begin{itemize}
\item[(a)] if $0 \leq \alpha < 1$, then
$$
\|u(t_{k})-u^{k}_{N}\|_{1} \leq \frac{T^{\alpha}}{1-\alpha}
\left(c_{u}\delta^{2-\alpha}+c \delta^{-\alpha}
N^{1-m}\|u\|_{L^{\infty}(H^{m})} \right), \quad k=1, \ldots, K,
$$
where $c_{u}$ is a constant depending on $u$;

\item[(b)] if $\alpha \rightarrow 1$, then
$$
\|u(t_{k})-u^{k}_{N}\|_{1} \leq  T \left(c_{u} \delta +c
\delta^{-1} N^{1-m}\|u\|_{L^{\infty}(H^{m})} \right), \quad k=1,
\ldots, K,
$$
\end{itemize}
where $\|u\|_{L^{\infty}(H^{m})}=\sup_{t \in (0, T)}
\|u(x,t)\|_{m}$, $c_{u}$ depends on $u$, and $c$ 
and $c_{u}$ are independent constants 
of $\delta$, $T$ and $N$.
\end{theorem}

\begin{proof}
(a) Let
$\widetilde{\varepsilon}_{N}^{k+1}=\pi_{N}^{1}u(t_{k+1})-u_{N}^{k+1}(x)$,
$\varepsilon_{N}^{k+1}= u(t_{k+1})-u_{N}^{k+1}$. We have
\begin{multline}
\label{eq23bis} 
\left( u(t_{k+1}) ,v\right) 
+\alpha_{0}\int_{\Omega}\nabla u(t_{k+1})   \nabla v \,dx
= (1-b_{1})(u(t_{k}), v)+ \sum_{j=1}^{k-1}(b_{j}-b_{j+1})(u(t_{k-j}), v)\\
+ b_{k}(u(t_{0}), v) + (r^{k+1}, v ) + \lambda \alpha_{0}
\left(\frac{f(u(t_{k+1}))}{\left( \int_{\Omega} f(u(t_{k+1}))\,dx
\right)^{2}}, v\right)
\end{multline}
for all $v\in H_{0}^{1}(\Omega)$. By the definition of the
projecting operator $\pi_{N}^{1}$ into $\mathbb{P}_{0}^{N}$, 
one has
\begin{equation}
\label{eq24bis}
\begin{split}
\left(\pi_{N}^{1}u(t_{k+1}) , v_{N}\right)
&+ \alpha_{0}\int_{\Omega}\nabla \pi_{N}^{1}u(t_{k+1})   \nabla v_{N} \,dx \\
&= (1-b_{1})(u(t_{k}), v_{N}) 
+\sum_{j=1}^{k-1}(b_{j}-b_{j+1})(u(t_{k-j}),
v_{N}) +b_{k}(u(t_{0}), v_{N})\\
&\qquad +( r^{k+1}, v ) + \lambda
\alpha_{0}\left(\frac{f(\pi_{N}^{1}u(t_{k+1}))}{ \left(
\int_{\Omega} f(\pi_{N}^{1}u(t_{k+1}))\,dx\right)^{2}}, v_{N}\right)
\end{split}
\end{equation}
for all $v_{N} \in \mathbb{P}_{N}^{0}(\Omega)$. Subtracting
\eqref{eq23bis} from \eqref{eq24bis}, we get, by taking 
$v_{N}=\widetilde{\varepsilon}_{N}^{k+1}$, using the triangular
inequality $\|{\varepsilon}_{N}^{k+1}\|_{1} \leq
\|\widetilde{\varepsilon}_{N}^{k+1}\|_{1}
+\|u(t_{k+1})-\pi_{N}^{1}u(t_{k+1})\|_{1}$ and following a standard
procedure as above, and using $\|r^{k+1}\|_{0} \leq c_{u}\delta^{2}$, that
\begin{equation}
\label{eq23} 
\|\varepsilon_{N}^{k+1}\|_{1}  \leq c \left(a_{k}
\|\varepsilon_{N}^{k}\|_{0}+\sum_{j=1}^{k-1}a_{k-j}
\|\varepsilon_{N}^{k-j}\|_{0}+a_{0} \|\varepsilon_{N}^{0}\|_{0} \right)\\
+ c_{u} \delta^{2}+ \|u(t_{k+1})-\pi_{N}^{1}u(t_{k+1})\|_{1}.
\end{equation}
On the other hand, using similar arguments, we can get
\begin{equation}
\label{eq24} 
\|\varepsilon^{j}_{N}\|_{1}
=\|u(t_{j})-u_{N}^{j}\|_{1} 
\leq c_{u} b_{j-1}^{-1} \delta^{2},
\quad j =1, 2, \ldots, K.
\end{equation}
The above inequality is obvious for $j=1$. 
Indeed, the error equation reads
$$
\left(\varepsilon_{N}^{1}, v_{N}\right)+ \alpha_{0} \left(\nabla
\varepsilon_{N}^{1},  \nabla v_{N}\right) 
= \left(\varepsilon_{N}^{0}, v_{N}\right) 
+\left(r^{1}, v_{N}\right)
= \left(r^{1}, v_{N}\right)
\quad \forall v_{N} \in H_{0}^{1}(\Omega).
$$
Letting $v_{N}= \varepsilon_{N}^{1}$, we have
$$
\|\varepsilon_{N}^{1}\|_{1}^{2} \leq \|r^{1}\|_{0}
\|\varepsilon^{1}_{N}\|_{0},
$$
which gives with \eqref{eq12bis} that
\begin{equation*}
\|\varepsilon^{1}_{N}\|_{1}  \leq c_{u} b_{0}^{-1} \delta^{2}.
\end{equation*}
Suppose now that \eqref{eq24} holds for all $j =1, 2, \ldots, k$.
We need to prove that it also holds for $j=k+1$. Similarly to the
above case, by combining  the corresponding equations of the exact
and discrete solutions and taking $v= \varepsilon^{k+1}_{N}$ as a
test function, it yields that
\begin{equation*}
\begin{split}
\|\varepsilon^{k+1}_{N}\|_{1}^{2}
&=\|\varepsilon^{k+1}_{N}\|_{0}^{2}+\alpha_{0}\|
\nabla \varepsilon^{k+1}_{N}\|_{0}^{2}\\
&\leq
(1-b_{1})\|\varepsilon^{k}_{N}\|_{0}\|\varepsilon^{k+1}_{N}\|_{0}
+\sum_{j=1}^{k-1}(b_{j}-b_{j+1})\| \varepsilon^{k-j}_{N}\|_{0} \|
\varepsilon^{k+1}_{N}\|_{0}
+ b_{k} \|\varepsilon^{0}_{N}\|_{0}\|\varepsilon^{k+1}_{N}\|_{0}\\
&\qquad + \| r^{k+1}\|_{0}\|\varepsilon^{k+1}_{N}\|_{0}
+ c\| \varepsilon^{k+1}_{N}\|_{1}^{2}\\
&\leq \left\{ (1-b_{1}) (c_{u}k \delta^{2})
+\sum_{j=1}^{k-1}(b_{j}-b_{j+1})( c_{u}(k-j) \delta^{2})
+ c_{u} \delta^{2} \right\} \|\varepsilon^{k+1}_{N}\|_{0}
+ c\| \varepsilon^{k+1}_{N}\|_{1}^{2}\\
& \leq \left\{ (1-b_{1}) \frac{k}{k+1} 
+\sum_{j=1}^{k-1}(b_{j}-b_{j+1})\frac{k-j}{k+1}+\frac{1}{k+1}
\right\}c_{u} (k+1)\delta^{2}\|\varepsilon^{k+1}_{N}\|_{0}
+ c\| \varepsilon^{k+1}_{N}\|_{1}^{2}\\
& \leq \left\{ (1-b_{1}) + \sum_{j=1}^{k-1}(b_{j}-b_{j+1})
-(1-b_{1})\frac{1}{k+1} -
\sum_{j=1}^{k-1}(b_{j}-b_{j+1})\frac{j+1}{k+1} +\frac{1}{k+1}\right\}\\
&\qquad \times c_{u} (k+1)\delta^{2}\|\varepsilon^{k+1}_{N}\|_{0}
+c\| \varepsilon^{k+1}_{N}\|_{1}^{2}.
\end{split}
\end{equation*}
Note that
\begin{equation*}
(1-b_{1})\frac{1}{k+1}+\sum_{j=1}^{k-1}(b_{j}-b_{j+1})\frac{j+1}{k+1}
+ b_{k} \geq \frac{1}{k+1} \left\{ (1-b_{1})+
\sum_{j=1}^{k-1}(b_{j}-b_{j+1}) +b_{k} \right\}= \frac{1}{k+1}.
\end{equation*}
It follows that
\begin{equation*}
\begin{split}
\|\varepsilon^{k+1}_{N}\|_{1}^{2}&  \leq \left\{ (1-b_{1})
+\sum_{j=1}^{k-1}(b_{j}-b_{j+1})+ b_{k} \right\} c_{u}
(k+1)\delta^{2}\|\varepsilon^{k+1}_{N}\|_{0} +c\|
\varepsilon^{k+1}_{N}\|_{1}^{2}.
\end{split}
\end{equation*}
Then, similar to the earlier development, one has
\begin{equation*}
(1-(c+\varepsilon))\| \varepsilon^{k+1}_{N}\|_{1}^{2} 
\leq \left(\left\{ (1-b_{1})+ \sum_{j=1}^{k-1}(b_{j}-b_{j+1}) 
+b_{k}\right\} c_{\varepsilon}c_{u}(k+1)\delta^{2} \right)^{2} 
=\left(c_{\varepsilon}c_{u}(k+1)\delta^{2}\right)^{2}.
\end{equation*}
It follows, for a well chosen $\varepsilon$ such that
$1-(c+\varepsilon) > 0$, that $\| \varepsilon^{k+1}_{N}\|_{1}
\leq c_{u}(k+1)\delta^{2}$. The estimate \eqref{eq24} is proved.
Applying \eqref{eq24} in \eqref{eq23} and using Lemma~\ref{lm2} gives
\begin{equation}
\label{eq25}
\begin{split}
\|\varepsilon_{N}^{k+1}\|_{1}  & \leq c \left(a_{k}
\|\varepsilon_{N}^{k}\|_{0}+\sum_{j=1}^{k-1}a_{k-j}
\|\varepsilon_{N}^{k-j}\|_{0}+a_{0}
\|\varepsilon_{N}^{0}\|_{0}\right)+ c_{u} \delta^{2}
+\|u(t_{k+1})-\pi_{N}^{1}u(t_{k+1})\|_{1}\\
& \leq  \left( a_{k} b_{k-1}^{-1}+\sum_{j=1}^{k-1}a_{k-1}
b_{k-j}^{-1} \right)
c_{u} \delta^{2} + c_{u} \delta^{2} + c N^{1-m}\|u(t_{k+1)}\|_{m}\\
& \leq \left( a_{k} +\sum_{j=1}^{k-1}a_{k-j}\right)
c_{u}b_{k}^{-1} \delta^{2}+ c_{u}  b_{k}^{-1}
b_{k}\delta^{2} +c N^{1-m}\|u(t_{k+1)}\|_{m}\\
& \leq \left( a_{k} +\sum_{j=1}^{k-1}a_{k-j} + b_{k} \right) c_{u}
b_{k}^{-1} \delta^{ 2}
+c N^{1-m}\|u(t_{k+1)}\|_{m}\\
& \leq  c_{u} b_{k}^{-1} \delta^{ 2} +c N^{1-m}\|u(t_{k+1)}\|_{m}.
\end{split}
\end{equation}
Using again $k^{-\alpha}b_{k-1}^{-1} \leq \frac{1}{1-\alpha}$ 
and $k\delta \leq T$, $k=1, 2, \ldots, K$, we have
\begin{equation*}
\begin{split}
\|\varepsilon_{N}^{k}\|_{1}  
& \leq  c_{u} b_{k-1}^{-1} \delta^{2}
+ c N^{1-m}\|u(t_{k)}\|_{m}\\
& \leq c_{u} k^{-\alpha}  b_{k-1}^{-1} k^{\alpha}\delta^{ 2}
+c \delta^{ \alpha} \delta^{ -\alpha} N^{1-m}\|u(t_{k)}\|_{m}\\
& \leq c_{u} (k^{-\alpha}  b_{k-1}^{-1})
(k\delta)^{\alpha}\delta^{ 2-\alpha}
+c (k\delta)^{\alpha} k ^{-\alpha}\delta^{ -\alpha} N^{1-m}\|u(t_{k)}\|_{m}\\
& \leq  (k^{-\alpha}  b_{k-1}^{-1})T^{\alpha} \left(c_{u} \delta^{
2-\alpha} + c \delta^{ -\alpha} N^{1-m}\|u\|_{L^{\infty}(H^{m})}\right)\\
& \leq \frac{T^{\alpha}}{1-\alpha} \left(c_{u} \delta^{ 2-\alpha}
+ c \delta^{ -\alpha} N^{1-m}\|u\|_{L^{\infty}(H^{m})} \right).
\end{split}
\end{equation*}
(b) Following the same lines as \eqref{eq24}, we have 
$$
\|u(t_{j})-u^{j}_{N}\|_{1} \leq c_{u} j\delta^{2},
\quad j=1, \ldots, K.
$$
Using the triangular inequality, we obtain
\begin{equation}
\label{eq27}
\begin{split}
\|\varepsilon_{N}^{k+1}\|_{1} & \leq c \left(a_{k}
\|\varepsilon_{N}^{k}\|_{0}+\sum_{j=1}^{k-1}a_{k-j}
\|\varepsilon_{N}^{k-j}\|_{0}+a_{0} \|\varepsilon_{N}^{0}\|_{0}
\right)+ c_{u} \delta^{2}
+\|u(t_{k+1})-\pi_{N}^{1}u(t_{k+1})\|_{1} \\
& \leq a_{k} (c_{u} k \delta^{2}) +\sum_{j=1}^{k-1}a_{k-j} (c_{u}
(k-j) \delta^{2}) + c_{u} \delta^{2}
+ cN^{1-m}\|u\|_{L^{\infty}(H^{m})} \\
& \leq \left( a_{k} \frac{k}{k+1}
+\sum_{j=1}^{k-1}a_{k-j}\frac{k-j}{k+1}+ \frac{1}{k+1}\right) c_{u}
(k+1)\delta^{2}+cN^{1-m}\|u\|_{L^{\infty}(H^{m})} \\
& \leq \left( a_{k} +\sum_{j=1}^{k-1}a_{k-j}-\frac{a_{k}}{k+1}
-\sum_{j=1}^{k-1}a_{k-j}\frac{j+1}{k+1}+\frac{1}{k+1} \right)
c_{u} (k+1)\delta^{2} +cN^{1-m}\|u\|_{L^{\infty}(H^{m})}.
\end{split}
\end{equation}
Since $k+1 \geq 1$, we easily see that
$$
a_{k}\frac{1}{k+1} 
+\sum_{j=1}^{k-1}a_{k-j}\frac{j+1}{k+1}+\frac{1}{k+1}+a_{0} 
\geq \frac{1}{k+1} \left(a_{k} + \sum_{j=1}^{k-1}a_{k-j}+a_{0} \right)
=\frac{1}{k+1}.
$$
Then,
$$
-\frac{a_{k}}{k+1} - \sum_{j=1}^{k-1}a_{k-j}
\frac{j+1}{k+1}+\frac{1}{k+1} \leq a_{0}.
$$
Injecting the above inequality into \eqref{eq27} gives
\begin{equation*}
\begin{split}
\|\varepsilon_{N}^{k+1}\|_{1}  & \leq \left( a_{k}
+\sum_{j=1}^{k-1}a_{k-j} +a_{0} \right) c_{u} (k+1)\delta^{2}
+cN^{1-m}\|u\|_{L^{\infty}(H^{m})}\\
& \leq c_{u} (k+1)\delta^{2}+ +cN^{1-m}\|u\|_{L^{\infty}(H^{m})}.
\end{split}
\end{equation*}
Therefore, we obtain that
\begin{equation*}
\begin{split}
\|\varepsilon_{N}^{k}\|_{1}  & \leq c_{u} k \delta^{2}
+cN^{1-m}\|u\|_{L^{\infty}(H^{m})}, \\
& \leq c_{u} T \delta + (c k \delta) (k \delta)^{-1}
N^{1-m}\|u\|_{L^{\infty}(H^{m})}\\
& \leq T \left( c_{u}  \delta + c \delta^{-1}
N^{1-m}\|u\|_{L^{\infty}(H^{m})} \right)
\end{split}
\end{equation*}
for all $1 \leq k \leq K$ such that $k \delta \leq T$. 
Hence, item $(b)$ of Theorem~\ref{thm6} is proved.
\end{proof}


\section{Conclusion}
\label{sec:conc}

We considered the problem of the nonlocal
time-fractional thermistor problem in the Caputo sense.
The main novelty was to use fractional derivatives 
to model memory effects. Main results include: 
a finite difference scheme for the temporal discretization 
of the problem; and a finite difference-Galerkin spectral 
method to obtain error estimates of fractional order convergence.
It should be mentioned that the Galerkin method is generally 
computationally expensive and difficult to extend to more complex geometries 
and higher spatial dimensions. Compared to a standard semilinear equation, 
the main challenge here is due to the nonstandard nonlocal nonlinearity 
on the right-hand side of the partial differential equation. 
For the existence of solution to the scheme, 
the Lax--Milgram theorem is not applicable due to the nonlocal term. 
The latter makes the calculus technical and cumbersome. Furthermore, 
for example Lemma~\ref{lm2} cannot be applied because 
of lack of regularity of the solution. 
The estimated errors obtained by our method 
depend strictly on the solution, which needs to be regular. 
Another difficulty is that the solution in a given time 
depends on the solutions of all previous time levels. 
Then, to compute the solution at the current time level, 
one needs to save all previous solutions. This fact 
makes the storage expensive. In the present context, 
numerical experiments are therefore an interesting
direction of future research. 


\section*{Acknowledgements}

This work has been partially supported by the Portuguese
Foundation for Science and Technology (FCT) within project
UID/MAT/04106/2013 (CIDMA). The authors are very grateful 
to two anonymous referees for their helpful comments, 
whose stimulating questions and remarks allowed 
the improvement of the paper.


\bigskip



\end{document}